\documentclass{amsart}
\usepackage{latexsym}
\usepackage{amsfonts,amsmath,amssymb,stmaryrd,amscd,amsxtra,amsthm,verbatim,calc}
\usepackage[all]{xy}

\RequirePackage[dvipsnames,usenames]{color}

\newtheorem{theorem}{Theorem}[section]
\newtheorem{lemma}[theorem]{Lemma}
\newtheorem{proposition}[theorem]{Proposition}
\newtheorem{corollary}[theorem]{Corollary}

\theoremstyle{definition}
\newtheorem{definition}[theorem]{Definition}
\newtheorem{example}[theorem]{Example}

\theoremstyle{remark}
\newtheorem{remark}[theorem]{Remark}

\numberwithin{equation}{subsection}

\DeclareMathOperator{\Tor}{Tor}

 \DeclareMathOperator{\Spec}{Spec}
\DeclareMathOperator{\depth}{depth}
\DeclareMathOperator{\pdim}{pdim}

\DeclareMathOperator{\fa}{\mathfrak{a}}
\DeclareMathOperator{\fm}{\mathfrak{m}}
\DeclareMathOperator{\fn}{\mathfrak{n}}

\DeclareMathOperator{\fp}{\mathfrak{p}}

\DeclareMathOperator{\cohodim}{cd}

\title{An extension of a theorem of Hartshorne}

\author[M.~Katzman]{Mordechai Katzman}
\address[Katzman]{Department of Pure Mathematics,
University of Sheffield, Hicks Building, Sheffield S3 7RH, United Kingdom}
\email{M.Katzman@sheffield.ac.uk}

\author[G.~Lyubeznik]{Gennady Lyubeznik}
\address[Lyubeznik]{School of Mathematics, University of Minnesota, 207 Church Street, Minneapolis, MN 55455, USA}
\email{gennady@math.umn.edu}

\author[W.~Zhang]{Wenliang Zhang}
\address[Zhang]{Department of Mathematics, University of Nebraska, Lincoln, NE 68588, USA}
\email{wzhang15@unl.edu}
\thanks{M.K.  gratefully acknowledges support from EPSRC grant EP/J005436/1;
G.L. was partially supported by NSF grant DMS \#1161783, and W.Z. by NSF grants DMS \#1247354/\#1405602 and an EPSCoR First Award grant.
G.L. and W.Z. were also supported by NSF grant 0932078000 while in residence at MSRI}

\begin{document}

\begin{abstract}
We extend a classical theorem of Hartshorne concerning the connectedness of the punctured spectrum of a local ring by analyzing the homology groups of a simplicial complex associated with the minimal primes of a local ring.
\end{abstract}

\maketitle
\section{introduction}
R. Hartshorne proved in \cite[Proposition~2.1]{HartshorneCompleteIntersectionConnectedness} that, if $(R,\fm)$ is a noetherian local ring whose depth is at least 2, then the punctured spectrum ({\it i.e.} $\Spec(R)\backslash \{\fm\}$) is connected. In this note we partially extend this result to the case depth $\geq 3$. We express our extension in terms of a simplicial complex (which already appeared in \cite[Theorem 1.1]{LyubeznikSomeLCModules}) defined as follows.

\begin{definition}
\label{definition: the simplicial complex Delta}
Let $R$ be a noetherian commutative ring and let $\{\fp_1,\dots,\fp_n\}$ be the minimal primes of $R$. Assume that $R$ is either local with maximal ideal $\mathfrak m$ or that $R$ is graded with $\mathfrak m$ the ideal of elements of positive degrees. A simplicial complex $\Delta(R)$ (or $\Delta$ whenever $R$ are clear from the context) is defined as follows: $\Delta(R)$ is the simplicial complex on the vertices $1,\dots,n$ such that a simplex $\{i_0,\dots,i_s\}$ is included in $\Delta(R)$ if and only if $\sqrt{\fp_{i_0}+\cdots +\fp_{i_s}}\neq \fm$.
\end{definition}

Note that the punctured spectrum of $R$ ({\it i.e.} $\Spec(R)\backslash \{\fm\}$) is connected (as a topological space under the Zariski topology) if and only if $\widetilde{H}_0(\Delta(R),G)=0$ for some (equivalently every) non-zero abelian group, where $\widetilde{H}_0(\Delta(R),G)$ is the $0$-th reduced singular homology of the simplicial complex $\Delta(R)$ with coefficients in $G$. Thus Hartshorne's theorem says that if depth$R\geq 2$, then $\widetilde{H}_0(\Delta(R),G)=0.$
Our main result is the following extension of Hartshorne's result.

\begin{theorem}
\label{theorem: depth 3 char p}
Let $(R,\fm,k)$ be a noetherian commutative complete local ring  of characteristic $p$. Assume that the residue field $k$ is separably closed and that $\depth(R)\geq 3$. Then
\[\widetilde{H}_0(\Delta(R);k)=\widetilde{H}_1(\Delta(R);k)=0.\]
\end{theorem}

We also have a graded analog of Theorem \ref{theorem: depth 3 char p} over any separably closed field (not just in characteristic $p$).
\begin{theorem}
\label{theorem: depth 3 graded case}
Let $R=\bigoplus_{i\in \mathbb{N}} R_i$ be a standard $\mathbb{N}$-graded ring over a separably closed field $k$ ({\it i.e.} $R_0=k$ and $R$ is finitely generated by $R_1$ as a $k$-algebra). If $\depth(R)\geq 3$, then
\[\widetilde{H}_0(\Delta(R);k)=\widetilde{H}_1(\Delta(R);k)=0.\]
\end{theorem}

Hartshorne's theorem provides a sufficient condition for $\depth(R/I)\leq 1$. Namely, if $\tilde H_0(\Delta(R);k)\ne 0$, then $\depth(R/I)\leq 1$. Our Theorem \ref{theorem: depth 3 graded case} provides a sufficient condition for $\depth(R/I)\leq 2$. Namely, if $\widetilde{H}_1(\Delta(R);k)\neq 0$, then $\depth(R/I)\leq 2$.

We also show by an example (Example \ref{Reisner's example}) that Theorem \ref{theorem: depth 3 graded case}
(and hence also Theorem \ref{theorem: depth 3 char p})
does not necessarily hold if one replaces $k$ in $\widetilde{H}_*(\Delta(R);k)$ with $\mathbb{Z}$.

We do not know whether the assumptions that $R$ is complete and the residue field is separably closed in Theorem \ref{theorem: depth 3 char p} or the assumption that $k$ is separably closed in Theorem \ref{theorem: depth 3 graded case} can be removed. Neither do we know whether the analogue of Theorem \ref{theorem: depth 3 char p} in characteristic 0 holds.

Finally, it would be very interesting to know whether there exists a similar connection between higher values of the depth of
$R$ and the vanishing of higher homology groups of the complex $\Delta(R)$. When $R$ is a Stanley-Reisner ring
we obtain a natural extension of  Theorem \ref{theorem: depth 3 graded case}  (Corollary \ref{Corollary: depth and homology in Stanley Reisner rings})
and show that if $\depth R \geq d$ then $\widetilde{H}_j \left( \Delta(R), k \right)=0$ for all $0\leq j\leq d-2$.

\section{Proof of Theorem \ref{theorem: depth 3 char p}}
In this section, we will prove our main technical result that implies Theorems \ref{theorem: depth 3 char p} and \ref{theorem: depth 3 graded case}. To this end, we recall the definition of the cohomological dimension.

\begin{definition}
Let $I$ be an ideal of a noetherian commutative ring $R$. Then the {\it cohomological dimension} of the pair $(R,I)$, denoted by $\cohodim(R,I)$, is defined by
\[\cohodim(R,I):=\max\{j|H^j_I(R)\neq 0\},\]
where $H^j_I(R)$ is the $j$th local cohomology module of $R$ supported at $I$.
\end{definition}

We will also consider the following spectral sequence (\cite[p.~39]{AM-GL-Z})
\begin{equation}
\label{MV spectral sequence}
E^{-a,b}_1=\bigoplus_{i_0<\cdots <i_a}H^b_{\fp_{i_0}+\cdots +\fp_{i_a}}(R) \Rightarrow H^{b-a}_{\fp_1\cap \cdots \cap \fp_n}(R)=H^{b-a}_I(R).
\end{equation}
where $R$ is a noetherian commutative ring, $I$ is an ideal of $R$, and $\{\fp_1,\dots,\fp_s\}$ is the set of minimal primes of $I$.

To prove Theorem \ref{theorem: depth 3 char p}, we will need the following results which are completely characteristic-free; they hold even in mixed-characteristic.

\begin{lemma}
\label{lemma: E2 and relative homology}
Let $I$ be an ideal of a $d$-dimensional noetherian complete local domain $(R,\fm,k)$ and let $\{\fp_1,\dots,\fp_n\}$ be the set of minimal primes of $I$. Then
\[E^{-t,d}_2\cong H_t(\Delta_n,\Delta(R/I);H^d_{\fm}(R)),\]
where $E^{-t,d}_2$ is the $E_2$-term of the spectral sequence (\ref{MV spectral sequence}), $\Delta_n$ denotes the full $n$-simplex, and $H_t(\Delta_n,\Delta(R/I);H^d_{\fm}(R))$ denotes the singular homology of the pair $(\Delta_n,\Delta(R/I))$ with coefficients in $H^d_{\fm}(R)$.
\end{lemma}
\begin{proof}
The spectral sequence (\ref{MV spectral sequence}) with $b=d$ gives us a complex
\begin{equation}
\label{complex: E1 page}
\cdots \to E^{-3,d}_1\to E^{-2,d}_1\to E^{-1,d}_1\to\cdots
\end{equation}
whose homology is $E^{*,d}_2$. By the Hartshorne-Lichtenbaum vanishing theorem (\cite[Theorem~3.1]{HartshorneCohomologicalDimension}), for each ideal $\fa$ of $R$, one has $H^d_{\fa}(R)=0$ if and only if $\sqrt{\fa}\neq \fm$. Consequently, the complex (\ref{complex: E1 page}) becomes
\begin{equation}
\label{complex: only top local cohomology surviving}
\cdots \to \bigoplus_{\substack{j_0<\cdots <j_t;\\ \sqrt{\fp_{j_0}+\cdots +\fp_{j_t}}=\fm}}H^d_{\fm}(R)\to \bigoplus_{\substack{i_0<\cdots <i_{t-1};\\ \sqrt{\fp_{i_0}+\cdots+\fp_{i_{t-1}}}=\fm}}H^d_{\fm}(R) \to \cdots
\end{equation}

Let $C^*(\Delta_n)$ and $C^*(\Delta(R/I))$ denote the complexes of singular chains of $\Delta_n$ and $\Delta(R/I)$ with coefficients in $H^d_{\fm}(R)$ respectively and let $\widetilde{C}^*$ denote the complex of (\ref{complex: only top local cohomology surviving}). One can see that the condition imposed on $\sqrt{\fp_1+\cdots+\fp_t}$ appearing in (\ref{complex: only top local cohomology surviving}) is the opposite to the one imposed in the definition of $\Delta(R/I)$; consequently, one has a short exact sequence
\[0\to C^*(\Delta(R/I))\to C^*(\Delta_n)\to \widetilde{C}^*\to 0,\]
{\it i.e.} $\widetilde{C}^*$ is the complex of singular chains of the pair $(C^*(\Delta_n), C^*(\Delta(R/I)))$ with coefficients in $H^d_{\fm}(R)$. Hence the homology of the complex (\ref{complex: only top local cohomology surviving}) is the relative homology $H_*(\Delta_n,\Delta(R/I);H^d_{\fm}(R))$, {\it i.e.}
\[E^{-t,d}_2\cong H_t(\Delta_n,\Delta(R/I);H^d_{\fm}(R)),\]
for each integer $t$.
\end{proof}

\begin{theorem}
\label{theorem: main technical theorem}
Let $I$ be an ideal of a $d$-dimensional noetherian complete local domain $(R,\fm,k)$. Assume that $\cohodim(R,I)\leq d-3$ and $\cohodim(R,\fp_j)\leq d-2$ for each minimal prime $\fp_j$ of $I$. Then
\[\widetilde{H}_0(\Delta(R/I);H^d_{\fm}(R))=\widetilde{H}_1(\Delta(R/I);H^d_{\fm}(R))=0\]
where $\widetilde{H}_*(\Delta(R/I);H^d_{\fm}(R))$ denotes the reduced singular homology of $\Delta(R/I)$ with coefficients in $H^d_{\fm}(R)$.
\end{theorem}
\begin{proof}
Assume that $I$ has $n$ minimal primes $\fp_1,\dots,\fp_n$, and we will consider the spectral sequence (\ref{MV spectral sequence}).

Grothedieck's Vanishing Theorem (\cite[6.1.2]{BrodmannSharpLocalCohomology}) asserts that $H^j_I(M)=0$ for any ideal $I$ of a $d$-dimensional noetherian commutative ring $A$, any $A$-module $M$, and any integer $j>d$. Therefore, for each $r\geq 2$, one has
\[E^{-2-r,d+r-1}_1=\bigoplus_{i_0<\cdots <i_{2+r}}H^{d+r-1}_{\fp_{i_0}+\cdots +\fp_{i_{2+r}}}(R)=0;\]
hence the (incoming) differential $E^{-2-r,d+r-1}_r\to E^{-2,d}_r$ is 0. Since $E^{0,d-1}_1=\bigoplus_jH^{d-1}_{\fp_j}(R)$ and $\cohodim(R,\fp_j)\leq d-2$, we have $E^{0,d-1}_1=0$ and hence the (outgoing) differential $E^{-2,d-1}_2\to E^{0,d-1}_2$ is 0. If $r\geq 3$, then $-2+r>0$ and clearly $E^{-2+r,d-r+1}_r=0$. This shows that, for $r\geq 2$, all incoming and outgoing differentials from $E^{-2,d}_r$ are 0; consequently $E^{-2,d}_2=E^{-2,d}_{\infty}$. On the other hand, since $\cohodim(R,I)\leq d-3$, we have $H^{d-2}_I(R)=0$. Therefore $E^{-2,d}_2=0$. Similarly, one can also show that $E^{-1,d}_2=0$. On the other hand, it follows from the Hartshorne-Lichtenbaum vanishing theorem (\cite[Theorem~3.1]{HartshorneCohomologicalDimension}) that $E^{0,d}_1=0$. Thus, $E^{0,d}_2=0$.

Therefore, it follows from Lemma \ref{lemma: E2 and relative homology} that
\[H_2(\Delta_n,\Delta(R/I);H^d_{\fm}(R))\cong E^{-2,d}_2=0\ {\rm and}\ H_1(\Delta_n,\Delta(R/I);H^d_{\fm}(R))\cong E^{-1,d}_2=0.\]
The long exact sequence of homology
\begin{align}
\cdots & \to H_j(\Delta_n;H^d_{\fm}(R))\to H_j(\Delta_n,\Delta(R/I);H^d_{\fm}(R))\to H_{j-1}(\Delta(R/I);H^d_{\fm}(R))\notag\\
&\to H_{j-1}(\Delta_n;H^d_{\fm}(R))\to \cdots\notag
\end{align}
implies that
\begin{align}
H_1(\Delta(R/I);H^d_{\fm}(R))&\cong H_1(\Delta_n;H^d_{\fm}(R))=0\notag\\
H_0(\Delta(R/I);H^d_{\fm}(R))&\cong H_0(\Delta_n;H^d_{\fm}(R))=H^d_{\fm}(R)\notag
\end{align}
where $H_1(\Delta_n;H^d_{\fm}(R))=0$ since $\Delta_n$ is contractible. Therefore, we have
\begin{align}
\widetilde{H}_1(\Delta(R/I);H^d_{\fm}(R))&=0\notag\\
\widetilde{H}_0(\Delta(R/I);H^d_{\fm}(R))&=0.\notag
\end{align}
This finishes the proof of our theorem.
\end{proof}

\begin{remark}
\label{remark: coefficients of homology group}
By the Universal Coefficient Theorem (\cite[Theorem~3A.3]{HatcherAlgTopBook}), for each topological space $X$ and an abelian group $G$, there is a short exact sequence
\[0\to H_i(X;\mathbb{Z})\otimes_{\mathbb{Z}}G \to H_i(X;G) \to \Tor_1(H_{i-1}(X;\mathbb{Z}),G)\to 0.\]
Since $H_0(X;\mathbb{Z})$ is always a free $\mathbb{Z}$-module, we have $\Tor_1(H_{0}(X;\mathbb{Z}),G)=0$ and hence $H_1(X;G)\cong H_1(X;\mathbb{Z})\otimes_{\mathbb{Z}}G$. In particular, with notation as in Theorem \ref{theorem: main technical theorem}, we have
\begin{align}
H_1(\Delta(R/I);H^d_{\fm}(R)) &\cong H_1(\Delta(R/I);\mathbb{Z})\otimes_{\mathbb{Z}}H^d_{\fm}(R)\notag\\
H_1(\Delta(R/I);k) &\cong H_1(\Delta(R/I);\mathbb{Z})\otimes_{\mathbb{Z}}k\notag
\end{align}
Under the assumptions of Theorem \ref{theorem: main technical theorem}, for each nonzero element $r\in R$, the short exact sequence $0\to R\xrightarrow{r}R\to R/(r)\to 0$ induces an exact sequence
\[\cdots \to H^d_{\fm}(R)\xrightarrow{r} H^d_{\fm}(R)\to H^d_{\fm}(R/(r))=0,\]
where $H^d_{\fm}(R/(r))=0$ since $\dim(R/(r))<d$. This shows that $rH^d_{\fm}(R)=H^d_{\fm}(R)$ for each nonzero element $r\in R$. Therefore,
\begin{enumerate}
\item when $R$ doesn't have any integer torsion ({\it i.e.} when $R$ contains $\mathbb{Q}$ or doesn't contain a field), $H_1(\Delta(R/I);\mathbb{Z})\otimes_{\mathbb{Z}}H^d_{\fm}(R)\neq 0$ if and only if $H_1(\Delta(R/I);\mathbb{Z})$ contains a copy of $\mathbb{Z}$;
\item when the characteristic of $R$ is $p$, $H_1(\Delta(R/I);\mathbb{Z})\otimes_{\mathbb{Z}}H^d_{\fm}(R)\neq 0$ if and only if $H_1(\Delta(R/I);\mathbb{Z})$ contains either a copy of $\mathbb{Z}$ or a copy of $\mathbb{Z}/p\mathbb{Z}$.
\end{enumerate}
Consequently, when $R$ contains a field, it is clear that $H_1(\Delta(R/I);\mathbb{Z})\otimes_{\mathbb{Z}}H^d_{\fm}(R)=0$ if and only if $H_1(\Delta(R/I);\mathbb{Z})\otimes_{\mathbb{Z}}k=0$, {\it i.e.} \[\widetilde{H}_1(\Delta(R/I);H^d_{\fm}(R))=0\Leftrightarrow H_1(\Delta(R/I);k)=0.\]
On the other hand, it should be clear that
\[\widetilde{H}_0(\Delta(R/I);H^d_{\fm}(R))=0\Leftrightarrow H_0(\Delta(R/I);k)=0\]
always holds.
\end{remark}

To prove Theorem \ref{theorem: depth 3 char p}, we also need the following result due to Peskine-Szpiro.

\begin{theorem}[Remarque on p.~386 in \cite{PeskineSzpiroDimensionProjective}]
\label{theorem: Peskine-Szpiro}
Let $R$ be a $d$-dimensional regular ring of characteristic $p$ and let $I$ be an ideal of $R$. Then
\[\cohodim(R,I)\leq d-\depth(R/I).\]
\end{theorem}

\begin{proof}[Proof of Theorem \ref{theorem: depth 3 char p}]
By the Cohen Structure Theorem, there is a complete regular local ring $(S,\fn,k)$ that maps onto $R$. Let $I$ denote the kernel of the surjection $S\to R$. Then it is straightforward to check that, if $\{\fp_1,\dots,\fp_n\}$ is the set of minimal primes of $R$, then the preimages of $\fp_i$, $\{\widetilde{\fp}_1,\dots,\widetilde{\fp}_n\}$, are the minimal primes of $I$ in $S$. It is also clear that
\[\sqrt{\fp_{i_0}+\cdots +\fp_{i_s}}\neq \fm \Leftrightarrow \sqrt{\widetilde{\fp}_{i_0}+\cdots +\widetilde{\fp}_{i_s}}\neq \fn\]
and consequently $\Delta(R)=\Delta(S/I)$.

Since each $\widetilde{\fp}_j$ is a prime ideal in $S$, the punctured spectrum of $S/\widetilde{\fp}_j$ is connected. Hence $\cohodim(S,\widetilde{\fp}_j)\leq \dim(S)-2$ by \cite[Corollaire~5.5]{PeskineSzpiroDimensionProjective}. Furthermore, since $\depth(S/I)\geq 3$, we know that $\cohodim(S,I)\leq \dim(S)-3$ according to Theorem \ref{theorem: Peskine-Szpiro}. Therefore, our conclusion follows directly from Theorem \ref{theorem: main technical theorem} and Remark \ref{remark: coefficients of homology group}.
\end{proof}

To prove Theorem \ref{theorem: depth 3 graded case}, the following result of Varbaro is needed.

\begin{theorem}[Theorem 3.5 in \cite{VarbaroCDgradedChar0}]
\label{theorem: Varbaro}
Let $I$ be an homogeneous ideal of $R=k[x_1,\dots,x_d]$ where $k$ is a field of characteristic 0. If $\depth(R/I)\geq 3$, then $\cohodim(R,I)\leq d-3$.
\end{theorem}

\begin{proof}[Proof of Theorem \ref{theorem: depth 3 graded case}]
Since $R$ is standard graded, one can write $R=k[x_1,\dots,x_n]/I$ for a homogeneous ideal $I$ of $S=k[x_1,\dots,x_n]$. Let $\fp_1,\dots,\fp_t$ be the minimal primes of $I$ in $S$. Then $\cohodim(S,\fp_j)\leq \dim(S)-2$ by \cite[Corollaire~5.5]{PeskineSzpiroDimensionProjective}, \cite[Corollary~2.11]{OgusLocalCohomologicalDimension}, and \cite[Theorem~2.9]{HunekeLyuVanishing}. And one has $\cohodim(S,I)\leq \dim(S)-3$ according to Theorem \ref{theorem: Peskine-Szpiro} (characteristic $p$) and Theorem \ref{theorem: Varbaro} (characteristic 0). Therefore, our conclusion follows directly from Theorem \ref{theorem: main technical theorem} and Remark \ref{remark: coefficients of homology group}.
\end{proof}

We conclude this section with the following corollary.

\begin{corollary}
\label{corollary: d-3 generators in regular rings}
Let $(R,\fm,k)$ be an equicharacteristic complete regular local ring whose residue field is separably closed. Assume that an ideal $I$ of $R$ can be generated by $d-3$ elements, where $d=\dim(R)$. Then
\[\widetilde{H}_0(\Delta(R/I);k)=\widetilde{H}_1(\Delta(R/I);k)=0.\]
\end{corollary}
\begin{proof}
Since $I$ can be generated by $d-3$ elements, $\cohodim(R,I)\leq d-3$. Same as in the proof of Theorem \ref{theorem: depth 3 char p}, we have $\cohodim(R,\fp)\leq d-2$ for each minimal prime of $I$. Therefore, our corollary follows directly from Theorem \ref{theorem: main technical theorem}.
\end{proof}

\begin{remark}
In \cite[Theorem~6]{FaltingsFormalFunctions}, it is proved that, {\it if an ideal $I$ of a $d$-dimensional complete local domain $(R,\fm)$ can be generated by $d-2$ elements, then $\Spec(R)\backslash \{\fm\}$ is connected}. One may consider Corollary \ref{corollary: d-3 generators in regular rings} as an extension of Faltings' connectedness theorem. Of course it would be very interesting to know whether the requirement that $R$ be regular can be removed or at least weakened.
\end{remark}


\section{An application to Stanley-Reisner rings}

Let $S$ be a polynomial ring $k[x_1, \dots, x_n]$ where $k$ is any field, and let $K$ be a simplicial complex
with vertices $\{ x_1, \dots, x_n \}$.
In this section we explore the properties of $\Delta(R)$  when $R$ is the Stanley-Reisner ring
$R=k[\Delta]=S/I(K)$. We then apply this to construct Example \ref{Reisner's example} that shows that  Theorem \ref{theorem: depth 3 char p} does not hold
if the reduced homology groups $\widetilde{H}_*(\Delta(R);-)$ are computed with coefficients in $\mathbb{Z}$.

We first recall a definition and a well known result.
\begin{definition}
\label{Definition: Nerve}
Let $K$ be a simplicial complex and let $K_1, \dots, K_t$ be subcomplexes of $K$.
The \emph{nerve of $K_1, \dots, K_t$}, denoted $\mathcal{N}(K_1, \dots, K_t)$,
is the simplicial complex with vertices $K_1, \dots, K_t$ and faces consisting of sets $\{ K_{i_1}, \dots , K_{i_\ell} \}$ such that
$K_{i_1}\cap \dots \cap K_{i_\ell}  \neq \emptyset$.
\end{definition}
The usefulness of this construction derives from the following theorem (cf.~\cite[Theorem 10.6]{Bjorner.Topolocial.Methods}.)
\begin{theorem}
\label{Theorem: homology of nerves}
With the notation above, assume further that $K_1\cup \dots \cup K_t=K$.
If every non-empty intersection $K_{i_1}\cap \dots \cap K_{i_\ell}$ is contractible,
$K$ and  $\mathcal{N}(K_1, \dots, K_t)$ are homotopy equivalent (and hence have same homology groups.)
\end{theorem}

We can now prove the main result of this section.
\begin{theorem}\label{Theorem: homotopy equivalence}
Let $S=k[x_1, \dots, x_n]$, let $K$ be a simplicial complex with vertex set $V=\{x_1, \dots, x_n \}$, and let
$I=I(K)$ be the Stanley-Reisner ideal corresponding to $K$.
Then $\Delta(S/I)$ is homotopy equivalent to $K$.
\end{theorem}
\begin{proof}
Let $K_1, \dots, K_t$ be the facets (i.e., maximal faces) of $K$ and note that since every intersection of facets is a face, Theorem
\ref{Theorem: homology of nerves} implies that $K$ and  $\mathcal{N}(K_1, \dots, K_t)$ are homotopy equivalent.

The minimal primes of $I$ are generated by  complements of facets of $K$ (cf.~\cite[Theorem 5.1.4]{BrunsHerzog})
thus $\Delta(S/I)$ is the simplicial complex whose vertices are the complements of facets of $K$
and whose faces consist of
$$\{ \{ V\setminus K_{i_1}, \dots,  V\setminus K_{i_\ell} \} \,|\, (V\setminus K_{i_1}) \cup \dots \cup (V\setminus K_{i_\ell}) \neq V \} .$$

We conclude the proof by exhibiting the isomorphism of simplicial complexes $\Phi: \mathcal{N}(K_1, \dots, K_t) \rightarrow K$
given by $\Phi(K_i)=V\setminus K_i$.
\end{proof}

\begin{example}
\label{Reisner's example}
Let $K$ be Reisner's six-point triangulation of the projective plane (cf.~\cite[Remark 3]{Reisner}) and let
$R=k(K)$ be the Stanley-Reisner ring associated to $K$.
\cite[Theorem 1]{Reisner}  shows that $R$ is Cohen-Macaulay if and only if the characteristic of $k$ is not $2$ thus
its depth of $R$ is 3 when the characteristic of $k$ is not 2.
On the other hand the previous theorem implies that $\widetilde{H}_1(\Delta(R);\mathbb{Z})=\widetilde{H}_1(K;\mathbb{Z})\neq 0$.
We conclude that the analogue of Theorem \ref{theorem: depth 3 graded case} with homology computed with integer coefficients does not hold.
\end{example}

We can now extend Theorem \ref{theorem: depth 3 graded case} for Stanley-Reisner rings as follows.
\begin{corollary}\label{Corollary: depth and homology in Stanley Reisner rings}
Let $S=k[x_1, \dots, x_n]$, let $K$ be a simplicial complex with vertex set $V=\{x_1, \dots, x_n \}$, and let
$I=I(K)$ be the Stanley-Reisner ideal corresponding to $K$.
If $\depth S/I \geq d$ then $\widetilde{H}_j \left( \Delta(S/I), k \right)=0$ for all $0\leq j\leq d-2$.
\end{corollary}
\begin{proof}
The Auslander-Buchsbaum theorem implies that
the condition $\depth S/I \geq d$ is equivalent to $\pdim_S S/I \leq n-d$.
This translates, using Hochster's formula (cf.~\cite{HochsterCombinatorics}, \cite[Corollary 5.12]{MillerSturmfels}),
to the vanishing of
$\widetilde{H}_{\# W -n +j -1} \left( K_W, k \right)$ for all $0\leq j\leq d-1$ and all $W\subseteq V$, where $K_W$ denotes
the restriction of $K$ to the subset of its vertices $W$.
In particular, this holds for $W=V$, giving
$\widetilde{H}_{j -1} \left( K, k \right)=0$ for all $0\leq j\leq d-1$.
Theorem \ref{Theorem: homotopy equivalence} now implies
$\widetilde{H}_{j -1} \left(  \Delta(S/I), k \right)=0$ for all $0\leq j\leq d-1$.
\end{proof}

We can also reinterpret Corollary \ref{corollary: d-3 generators in regular rings} in the case of Stanley-Reisner ideals as follows.
Let $K$ be a simplicial complex with vertex set $V=\{x_1, \dots, x_n \}$, let $R=k[x_1, \dots, x_n]$ and let
$I=I(K)$ be the Stanley-Reisner ideal corresponding to $K$. The ideal $I$ is generated by $n-t$ elements if and only if the Alexander dual $K^*$ has at most
$n-t$ facets. Also $\widetilde{H}_{j}(\Delta(R/I), k)=\widetilde{H}_{j}(K, k)= \widetilde{H}_{n-3-j}(K^*, k)$ where the first equality follows from Theorem
\ref{Theorem: homotopy equivalence} and the second from Alexander Duality (cf.~\cite[Theorem 5.6]{MillerSturmfels}).
Now Corollary \ref{corollary: d-3 generators in regular rings} implies the following:
if a simplicial complex $\Delta$ with $n$ vertices has at most $n-3$ facets, then $\widetilde{H}_{n-3}(\Delta, k)=\widetilde{H}_{n-4}(\Delta, k)=0$.
This also leads to the following natural generalization of Corollary \ref{corollary: d-3 generators in regular rings}.

\begin{proposition}\label{Proposition: vanishing homology with few facets}
For a simplicial complex $\Delta$ with $n$ vertices and at most $\mu$ facets,
$\widetilde{H}_{i}(\Delta, k)=0$ for all $i\geq \mu-1$.
\end{proposition}
\begin{proof}
Let $K_1\cup \dots \cup K_s$ be the distinct facets of $K$ and
let $\mathcal{N}=\mathcal{N}(K_1, \dots, K_s)$.
Note that $\dim \mathcal{N}\leq s-1$ and that $\mathcal{N}$ is $(s-1)$-dimensional precisely when it is a simplex.
An application of Theorem \ref{Theorem: homology of nerves}
gives $\widetilde{H}_{i}(\Delta, k)= \widetilde{H}_{i}( \mathcal{N}, k)$.
and this vanishes for
all $i\geq s-1$, and since $s\leq \mu$, $\widetilde{H}_{i}(\Delta, k)= 0$ for all $i\geq \mu-1$.
\end{proof}

\begin{corollary}
\label{corollary: generalized d-3 generators in regular rings}
Let $K$ be a simplicial complex with vertex set $V=\{x_1, \dots, x_n \}$, let $R=k[x_1, \dots, x_n]$ and let
$I=I(K)$ be the Stanley-Reisner ideal corresponding to $K$. If $I$ is generated by $n-t$ elements then
$\widetilde{H}_i(\Delta(R/I);k)=0$
for all $0\leq i \leq t-2$.
\end{corollary}
\begin{proof}
Using the previous discussion we reduce the problem to showing that
$\widetilde{H}_{n-3-j}(K^*, k)$ vanishes for $0\leq j\leq t-2$ and this follows from
the application of Proposition \ref{Proposition: vanishing homology with few facets} to $K^*$.
\end{proof}

\section*{Acknowledgements}
We thank Vic Reiner for pointing out to us Theorem \ref{Theorem: homology of nerves}
and its relevance to Example \ref{Reisner's example}. We also thank John Shareshian for pointing out
Proposition  \ref{Proposition: vanishing homology with few facets} to us.

\bibliographystyle{skalpha}
\bibliography{CommonBib}

\end{document}